\documentclass{amsart}

\usepackage{amscd,amssymb,graphics}
\usepackage{hyperref}
\usepackage{amsfonts}
\usepackage{amsmath}
\usepackage{enumerate}
\usepackage{tikz}

\input xy
\xyoption{all}

\usepackage{vmargin}
\setmarginsrb{20mm}{10mm}{20mm}{13mm}%
          {10mm}{7mm}{10mm}{10mm}

\newtheorem{thm}{Theorem}[section]
\newtheorem{fact}[thm]{Fact}
\newtheorem{corol}[thm]{Corollary}
\newtheorem{lemma}[thm]{Lemma}
\newtheorem{prop}[thm]{Proposition}
\newtheorem{conj}[thm]{Conjecture}

\newtheorem{defi}[thm]{Definition}
\numberwithin{equation}{section}
\theoremstyle{remark}
\newtheorem{remark}[thm]{Remark}
\newtheorem{example}[thm]{Example}
\newtheorem{problem}[thm]{Problem}

\newcommand{\ben}{\begin{enumerate}}
\newcommand{\een}{\end{enumerate}}

\def\R {{\mathbb R}}
\def\Q {{\mathbb Q}}

\def\N{{\mathbb N}}
\def\T{{\mathbb T}}
\def\Z {{\mathbb Z}}

\def\End{\operatorname{End}}
\def\Aut{\operatorname{Aut}}

\begin{document}	
	\title[Algebraic entropy on strongly compactly covered groups]{Algebraic entropy on strongly compactly covered groups}
	\author[A. Giordano Bruno]{Anna Giordano Bruno}
	\address[A. Giordano Bruno]{\hfill\break
		Dipartimento di Matematica e Informatica
		\hfill\break
		Universit\`{a} di Udine
		\hfill\break
		Via delle Scienze  206, 33100 Udine
		\hfill\break
		Italy}
	\email{anna.giordanobruno@uniud.it} 
	\author[M. Shlossberg]{Menachem Shlossberg}
	\address[M. Shlossberg]{\hfill\break
		Dipartimento di Matematica e Informatica
		\hfill\break
		Universit\`{a} di Udine
		\hfill\break
		Via delle Scienze  206, 33100 Udine
		\hfill\break
		Italy}
	\email{menachem.shlossberg@uniud.it}
	\author[D. Toller]{Daniele Toller}
	\address[D. Toller]{\hfill\break
		Dipartimento di Matematica e Informatica
		\hfill\break
		Universit\`{a} di Udine
		\hfill\break
		Via delle Scienze  206, 33100 Udine
		\hfill\break
		Italy}
	\email{daniele.toller@uniud.it}

\maketitle


\begin{abstract}
We introduce a new class of locally compact groups, namely the strongly compactly covered groups, which are the Hausdorff topological groups $G$ such that every element of $G$ is contained in a compact open normal subgroup of $G$. For continuous endomorphisms $\phi:G\to G$ of these groups we compute the algebraic entropy and study its properties. Also an Addition Theorem is available under suitable conditions.
\end{abstract}

\section{Introduction}

The topological entropy for continuous selfmaps of compact spaces was introduced by Adler, Konheim and McAndrew \cite{AKM}, in analogy with the measure entropy studied in Ergodic Theory by Kolmogorov and Sinai. Later on, Hood \cite{hood} extended Bowen-Dinaburg's entropy (see \cite{B,Din}) to uniformly continuous selfmaps of uniform spaces. This notion of entropy coincides with the topological entropy from \cite{AKM} in the compact case, and it can be considered in particular for continuous endomorphisms $\phi$ of locally compact groups $G$ (see \cite{GBV}).

Also the algebraic entropy has its roots in the paper \cite{AKM}, where it was considered for endomorphisms of (discrete) abelian groups. Weiss \cite{W} studied this concept in the torsion case; in particular, he found the precise connection of the algebraic entropy to the topological entropy (and also to the measure entropy) by means of Pontryagin duality, via a so-called Bridge Theorem. For a recent fundamental paper on the algebraic entropy for endomorphisms of torsion abelian groups we refer the reader to \cite{DGSZ}.  
 
Peters \cite{Pet} defined the algebraic entropy differently with respect to Weiss and his definition was restricted to automorphisms of discrete abelian groups. Peters' entropy need not vanish on torsion-free abelian groups unlike the algebraic entropy from \cite{AKM}. At the same time, these two notions of algebraic entropy coincide on automorphisms of torsion abelian groups.  In \cite{DG2} Peters' definition was appropriately modified to introduce algebraic entropy for endomorphisms of discrete abelian groups and all the fundamental properties of algebraic entropy were extended to this general setting; moreover, a connection was found between the algebraic entropy and Lehmer's problem from Number Theory based on the so-called Algebraic Yuzvinski Formula (see \cite{GBV-yuz,GBV-yuzapp}). Later on, also the Bridge Theorem was extended to all discrete abelian groups in \cite{DG-bt}.

Peters \cite{Pet1} extended his definition of algebraic entropy from \cite{Pet} to topological automorphisms of  locally compact abelian groups, and Virili \cite{V} modified this definition (in a similar manner as in \cite{DG2}) to endomorphisms of locally compact abelian groups. A Bridge Theorem is available for topological automorphisms of locally compact abelian groups (see \cite{Pet1,V-BT}) and for continuous endomorphisms of compactly covered locally compact abelian groups (see \cite{DGB}). Recall that a topological group $G$ is called \emph{compactly covered} if each element of $G$ is contained in some compact subgroup of $G$. 

The commutativity of the groups in Virili's definition can be omitted as it was described in \cite{DG-islam}. We give now the general definition.

In this paper we consider always Hausdorff topological groups. For a topological group $G$, we denote by $\mathcal C(G)$ the family of all compact neighborhoods of the identity element $e_G$ of $G$, and by $\End (G)$ the set of continuous group endomorphisms of $G$. As usual we denote by $\N$ and $\N_+$ the natural numbers and the positive integers respectively.

Let $G$ be a locally compact group and  $\mu$ be a right Haar measure on $G$. For $\phi \in \End (G)$, a subset $U \subseteq G$, and $n\in\N_+$, the $n$-th $\phi$-trajectory of $U$ is
\[T_n(\phi,U) = U \cdot \phi(U)  \cdot\ldots\cdot  \phi^{n-1} (U).\] 
When $U \in \mathcal C(G)$, the subset $T_n(\phi,U)$ is also compact, so it has finite measure. 
The \emph{algebraic entropy of $\phi$ with respect to  $U$} is
\begin{equation}\label{def:H:alg}
H_{alg} (\phi, U) = \limsup_{n \to \infty} \frac{\log \mu (T_n(\phi,U))}{n},
\end{equation}
and it does not depend on the choice of the Haar measure $\mu$ on $G$.
The \emph{algebraic entropy of $\phi$} is 
\begin{equation*}
h_{alg}(\phi) = \sup \{ H_{alg} (\phi, U) : U \in \mathcal C(G) \}.
\end{equation*}
Note that $h_{alg}$ vanishes on compact groups.

\medskip
In this paper we study the algebraic entropy of continuous endomorphisms $\phi \in \End(G)$ of locally compact groups $G$ satisfying one of the equivalent conditions stated in Proposition \ref{cclca:structure} below; we call these groups \emph{strongly compactly covered}. 

First we recall that an \emph{FC-group} is a group in which every element has finitely many conjugates. It is known that a torsion group is an FC-group if and only if each of its finite subsets is contained in a finite normal subgroup (see \cite[14.5.8]{R}); for this reason the torsion FC-groups are also called {\it locally finite and normal}. In particular, torsion FC-groups are locally finite.

Given a topological group $G$, we denote by $\mathcal B(G)$ and $\mathcal N(G)$ the subfamilies of $\mathcal C(G)$ of all compact open subgroups and all compact open normal subgroups of $G$, respectively. Clearly, $\mathcal N(G) \subseteq \mathcal B(G) \subseteq \mathcal C(G)$.

\begin{prop}\label{cclca:structure}
Let $G$ be a topological group. Then the following conditions are equivalent (and define $G$ to be a strongly compactly covered group):
\ben
	\item  $G$ is a locally compact group and $\mathcal N(G)$ is cofinal in $\mathcal C(G)$;
	\item there exists $K \in \mathcal N(G)$ such that $G/K$ is a torsion FC-group;
	\item for every $g\in G$ there exists $N_g\in \mathcal N(G)$ such that $g\in N_g$.
\een
\end{prop}

Looking at condition (3) in Proposition \ref{cclca:structure}, it is clear that a strongly compactly covered group is locally compact and compactly covered.
In view of condition (2), a strongly compactly covered group can also be called compact-by-(discrete torsion FC-group).
In particular, compact groups are strongly compactly covered.

We prove Proposition \ref{cclca:structure} in Section \ref{sccg:section}, where we also discuss several properties of strongly compactly covered groups. In particular (see Lemma \ref{cofinality:discrete}), we have that the strongly compactly covered discrete groups are precisely the torsion FC-groups, while the compactly covered discrete groups are precisely the torsion groups. As a consequence, we see that the class of strongly compactly covered groups is strictly contained in that of locally compact compactly covered groups, as for example there exist torsion groups that are not FC-groups.

\smallskip

On the other hand, an abelian topological group $G$ is strongly compactly covered precisely when it is locally compact and compactly covered (see Corollary \ref{cor:cclca}); moreover, a locally compact abelian group is compactly covered precisely when its Pontryagin dual group is totally disconnected (see Remark \ref{cc-td}).
So, some examples of strongly compactly covered abelian groups are: 
\ben
	\item compact abelian groups;
	\item locally compact  torsion abelian  groups (in particular, finite abelian groups);
	\item the $p$-adic numbers $\Q_p$.
\een

Now recall the following property, that motivated us to study the algebraic entropy for the class of strongly compactly covered groups, since it is clear from the definition (see also Remark \ref{rem:cofin}) that when computing the algebraic entropy it suffices to consider any cofinal subfamily of $\mathcal C (G)$. If $G$ is a  topological abelian group, then $\mathcal B(G) = \mathcal N(G)$.

\begin{fact}\cite[Proposition 2.2]{DGB}\label{BG:cofinal}
If $G$ is a compactly covered locally compact abelian group, then $\mathcal B(G)$ is cofinal in $\mathcal C(G)$.
\end{fact}

Generalizing the measure-free formula given in \cite{DG-islam} in the abelian case, in Lemma \ref{lem:conv} we prove that if $G$ is a locally compact group, then the algebraic entropy of $\phi\in \End(G)$ with respect to $U\in\mathcal N(G)$  is
 \[H_{alg}(\phi,U)=\lim_{n\to \infty}\frac{\log [T_n(\phi,U) :U] }{n}.\]
In case $G$ is strongly compactly covered, since $\mathcal N(G)$ is cofinal in $\mathcal C(G)$ by Proposition \ref{cclca:structure}(1), we obtain
\begin{equation} \label{eq:algnormal}
h_{alg}(\phi) = \sup \{ H_{alg} (\phi, U) : U \in \mathcal N(G) \}.
\end{equation}
These formulas allow for a simplified computation of the algebraic entropy for strongly compactly covered groups.

\medskip

In Section \ref{basic-sec} we study the basic properties of the algebraic entropy for continuous endomorphisms of strongly compactly covered groups: Invariance under conjugation, Logarithmic Law, weak Addition Theorem, Monotonicity for closed subgroups and Hausdorff quotients. 

We see that the algebraic entropy of the identity automorphism of a strongly compactly covered group is zero. This is not always true in the non-abelian case, in fact the identity automorphism of a finitely generated group of exponential growth has infinite algebraic entropy -- see \cite{DG-islam,GBSp}.

As a fundamental example we compute the algebraic entropy of the shifts, showing in particular that Invariance under inversion is not available in general for topological automorphisms $\phi$ of strongly compactly covered groups. Indeed, we find the precise relation between $h_{alg}(\phi)$ and $h_{alg}(\phi^{-1})$ by using the \emph{modulus of $\phi$} (see Proposition \ref{h:alg:phi:inverse}).

\medskip
The so-called Limit-free Formulas provide a way to simplify the computation of both the algebraic and the topological entropy in different contexts. For example, Yuzvinski \cite{Y} provided such a  formula for the algebraic entropy of endomorphisms of discrete torsion abelian groups. However, it turned out that  his formula holds true for injective endomorphisms \cite{DG-lf, DSV} even though it is not true in general (see \cite[Example 2.1]{DG-lf}).  
A general Limit-free Formula for the algebraic entropy of locally finite groups is given in \cite{DG-lf}, where also its counterpart for the topological entropy for totally disconnected compact groups is considered. A Limit-free Formula for the topological entropy for continuous endomorphisms of totally disconnected locally compact groups is given in \cite{GBV}, extending a result from \cite{GB}.

Inspired by \cite{CGB, GBV}, we find in Proposition \ref{prop:limit-free} a Limit-free Formula for the algebraic entropy that can be described as follows. For a locally compact group $G$, $\phi\in \End(G)$ and $U\in \mathcal N(G)$,  we prove that 
\[H_{alg} (\phi, U)=[U^- :\phi^{-1}U^-],\]  
where $U^-$ is the smallest (open) subgroup of $G$ containing $U$ and inversely $\phi$-invariant (i.e., satisfying $\phi^{-1}U^- \leq U^-$). 
Note that $U^-$ is defined in \cite{CGB} in the setting of locally linearly compact vector spaces, ``dualizing'' a similar construction used for the topological entropy in \cite{CGB,GBV}, which was inspired by ideas of Willis \cite{WIL}. 

\medskip
In Section \ref{BT-sec} we compare our formula stated above with the Limit-free Formula for the topological entropy of continuous endomorphisms of  totally disconnected locally compact groups obtained in \cite{GBV}. This allows us to produce an alternative proof (see Theorem \ref{BT}) for the Bridge Theorem  given in \cite{DGB}.

\medskip
The main property of entropy functions is the so-called Addition Theorem, that means additivity of the entropy function (see Equation \eqref{ATeq} below). It is also known as Yuzvinski's addition formula, since it was first proved by Yuzvinski for the topological entropy in the case of separable compact groups \cite{Y}. Later on, Bowen proved in \cite[Theorem 19]{B} a version of the Addition Theorem for compact metric spaces, while the general statement for compact groups is deduced from the metrizable case by Dikranjan and Sanchis in \cite[Theorem 8.3]{Dik+Manolo}. The Addition Theorem plays a fundamental role also in the Uniqueness Theorem for the topological entropy in the category of compact groups provided by Stoyanov \cite{St}.

The Addition Theorem for the topological entropy was recently extended to continuous endomorphisms of totally disconnected locally compact groups in \cite{GBV} under suitable assumptions, in particular it holds for topological automorphisms of totally disconnected locally compact groups.

A first Addition Theorem for the algebraic entropy was proved in \cite[Theorem 3.1]{DGSZ} for the class of discrete torsion abelian groups, and  was later generalized  to the class of discrete abelian groups in \cite[Theorem 1.1]{DG2}. Note that a discrete abelian group is strongly compactly covered precisely when it is torsion.

On the other hand, it is known \cite{GBSp} that the Addition Theorem does not hold in general for the algebraic entropy even for discrete solvable groups (while its validity for nilpotent groups is an open problem). This comes from the strict connection of the algebraic entropy with the group growth from Geometric Group Theory (see \cite{DG-islam,DG-pc,GBSp2}).

Let $G$ be a locally compact group, $\phi\in \End(G)$ and $H$ a closed normal $\phi$-invariant (i.e., satisfying $\phi(H)\leq H$)  subgroup of $G$. We say that the Addition Theorem holds for the algebraic entropy if 
\begin{equation}\label{ATeq}
h_{alg} (\phi) =h_{alg} (\bar \phi) + h_{alg} (\phi \restriction_H),
\end{equation}
where $\phi \restriction_H$ is the restriction of $\phi$ to $H$ and $\bar \phi : G/H \to G/H$ is the induced map on the quotient group.

We consider the following general problem. Note that it is not even known whether the Addition Theorem holds in general for locally compact abelian groups (see \cite{DG-islam,DGB}).

\begin{problem}
For which locally compact groups does the Addition Theorem hold?
\end{problem}

As our main result, we prove in Section \ref{sec:AT} that  the Addition Theorem holds for a strongly compactly covered group $G$  in case $H$ is a closed normal $\phi$-stable (i.e., $\phi(H)= H$) subgroup of $G$ with $\ker\phi\leq H$. 
In particular we find that the Addition Theorem holds for topological automorphisms of strongly compactly covered groups (see Corollary \ref{ATaut}).

In the discrete case this means that the Addition Theorem holds for automorphisms of torsion FC-groups (this is a first extension of \cite[Theorem 3.1]{DGSZ} to some non-abelian groups), and we conjecture that the same result can be extended to all locally finite groups (see Conjecture \ref{conj}).

Another consequence of our Addition Theorem (see Corollary \ref{c(G)}) is that, to compute the algebraic entropy of a topological automorphism $\phi$ of a strongly compactly covered group $G$, one can assume $G$ to be totally disconnected; indeed, in this case the connected component $c(G)$ of $G$ is $\phi$-stable and $h_{alg} (\phi) =h_{alg} (\bar \phi)$, where $\bar \phi : G/c(G) \to G/c(G)$ is a topological automorphism.

\section{Strongly compactly covered groups}\label{sccg:section}

We start this section by proving Proposition \ref{cclca:structure} which characterizes the strongly compactly covered groups. 

\begin{proof}[\bf Proof of Proposition \ref{cclca:structure}]
(1)$\Rightarrow$(2) Assume first that $G$ is a locally compact group with $\mathcal N(G)$  cofinal in $\mathcal C(G)$, and let $H\in \mathcal C(G)$. By the cofinality assumption, there exists $K\in \mathcal N(G)$ containing $H$. 
	
Let  $X$ be a finite subset of the discrete group $G/K$ containing the identity element, and consider the canonical projection $q:G\to G/K$. Then, $q^{-1} (X)\in \mathcal C(G)$, being a finite union of cosets of the compact open subgroup $K$, so there exists $M\in \mathcal N(G)$ such that $q^{-1} (X)\subseteq M$. It follows that $X\subseteq q(M)$, and $q(M)$ is a finite normal subgroup of $G/K$, being compact in $G/K$ discrete. This proves that $G/K$ is a torsion FC-group by Lemma \ref{cofinality:discrete}(3).
	
(2)$\Rightarrow$(3) Assume that there exists $K \in \mathcal N(G)$ such that $G/K$ is a torsion FC-group,  let $g \in G$  and $n \in \N_+$ be the order of $gK$ in $G/K$. Then $\langle g \rangle \cap K = \langle g^n \rangle$, so $\langle K,g \rangle /K \cong \langle g \rangle / \langle g^n \rangle$ is finite. As $G/K$ is a torsion FC-group there exists a finite normal subgroup $S$ of $G/K$ containing $\langle K,g \rangle /K$. It follows that $g\in N_g=q^{-1}(S)\in \mathcal N(G)$, where $q:G\to G/K$ is the canonical projection.	

(3)$\Rightarrow$(1) First of all, $G$ is locally compact  as $N_{e_G}\in \mathcal N(G)$.

	Now fix $C\in\mathcal C(G)$ and let us show that there exists $N\in \mathcal N(G)$ such that $C\subseteq N$. By our assumption for every $c\in C$ there exists $N_c\in \mathcal N(G)$ with $c\in N_c$. The compactness of $C$ implies that there exist finitely many $c_1, \ldots, c_n \in C$ such that 
	$C\subseteq \bigcup_{i=1}^n N_{c_i}\subseteq N_{c_1}\cdots N_{c_k}$. So take $N=N_{c_1}\cdots N_{c_k}\in \mathcal N(G)$.
\end{proof}

The following is a direct consequence of Proposition \ref{cclca:structure} for abelian groups. The implication (4)$\Rightarrow$(1) is Fact \ref{BG:cofinal}, while (3)$\Rightarrow$(4) is trivial.
 
\begin{corol}\label{cor:cclca} 
	Let $G$ be an abelian topological group. Then the following conditions are equivalent: \ben
	\item $G$ is a locally compact group and $\mathcal B(G)$ is cofinal in $\mathcal C(G)$;
	\item there exists $K \in \mathcal B(G)$ such that $G/K$ is  torsion;
	\item for every $g\in G$ there exists $N_g\in \mathcal B(G)$ such that $g\in N_g$;
	\item $G$ is a locally compact compactly covered group.
	\een
In particular, $G$ is strongly compactly covered if and only if $G$ is locally compact and compactly covered.
\end{corol}

By Corollary \ref{cor:cclca}, a locally compact abelian group $G$ is compactly covered precisely when the family $\mathcal N(G)=\mathcal B(G)$ is cofinal in $\mathcal C(G)$.
It is worth noting that in case $G$ is non-abelian, $\mathcal B(G)$ need not be cofinal in $\mathcal C(G)$, even if $G$ is compactly covered. 
This follows for example from Lemma \ref{cofinality:discrete} below, taking into account that there exist (necessarily non-abelian) discrete torsion groups which are not locally finite.

For a discrete group $G$, note that: 
\ben
\item $\mathcal C(G) = \{X\subseteq G: X\ \text{finite}, e_G\in X \}$;
\item $\mathcal B(G) = \{H\leq G: H\ \text{finite} \}$;
\item $\mathcal N(G) = \{N\leq G: N\ \text{finite normal} \}$.
\een

\begin{lemma}\label{cofinality:discrete}
Let $G$ be a discrete  group. Then:
\ben
\item $G$ is compactly covered if and only if it is torsion;
	\item  $\mathcal B(G)$ is cofinal in $\mathcal C(G)$ if and only if  $G$ is locally finite;
	\item  $\mathcal N(G)$ is cofinal in $\mathcal C(G)$ (i.e., $G$ is strongly compactly covered) if and only if $G$ is a torsion FC-group.
\een
\end{lemma}
\begin{proof}
(1) This equivalence follows from the definitions.

(2) Let $G$ be a locally finite group and $X\in \mathcal C(G)$. As $G$ is locally finite, the finite subset $X$ generates a finite subgroup $F$. Therefore, $X\subseteq F\in \mathcal B(G)$, and this proves the cofinality of $\mathcal B(G)$ in $\mathcal C(G)$.
	
Conversely, assume that $G$  is not locally finite. So, there exists a finite subset $F$ of $G$ which generates an infinite subgroup. Without loss of generality, we may assume that $e_G\in F,$  so that $F\in \mathcal C(G).$ On the other hand, no  finite subgroup $K$ of $G$ contains $F$.
	
(3) This immediately follows from the equivalence between the definition of locally finite and normal and the property of being a torsion FC-group (see \cite[14.5.8]{R}).
\end{proof}

As a consequence of the above lemma and in view of Proposition \ref{cclca:structure}, we obtain that the torsion FC-groups are the strongly compactly covered groups with respect to the discrete topology.

It is easy to see that the class of discrete torsion FC-groups contains all finite groups and  is stable under taking subgroups, quotients or forming direct sums. Now we extend this result to the larger class of strongly compactly covered groups.

\begin{lemma}\label{cor:cclcstable} 
The class of strongly compactly covered groups is stable under taking closed subgroups and Hausdorff quotients. 
\end{lemma}
\begin{proof}
Let $G$ be a strongly compactly covered group, and let $K \in \mathcal N(G)$ be such that $G/K$ is a torsion FC-group, by Proposition \ref{cclca:structure}. 

If $H$ is a closed subgroup of $G$, then $H\cap K \in \mathcal N(H)$.
Since $H/H\cap K\cong HK/K\leq G/K$ we deduce that $H/H\cap K$ is a torsion FC-group, so $H$ is strongly compactly covered by Proposition \ref{cclca:structure}.
 
Assuming that $H$ is also normal, one has $KH/H\in \mathcal N(G/H)$, being the image of $K$ in $G/H$. As the class of discrete torsion FC-groups is closed under taking quotients, we complete the proof using  the isomorphisms  $$(G/H)/ (KH/H)\cong G/KH\cong (G/K)/(KH/K)$$ and again applying Proposition \ref{cclca:structure}.
\end{proof}

Now we check that the class of strongly compactly covered groups is stable also under taking finite products. 

\begin{lemma}\label{lem:cofprod}
Let $G_1,G_2$ be strongly compactly covered groups. Then $$\mathcal F = \{U_1\times U_2: U_i\in \mathcal N(G_i)\}\subseteq \mathcal N(G_1\times G_2)$$ are cofinal subfamilies of $\mathcal C (G_1\times G_2)$.	

In particular, $G_1 \times G_2$ is a strongly compactly covered group.
\end{lemma}
\begin{proof}
The group $G_1 \times G_2$ is locally compact, and it suffices to prove that $\mathcal F$ is cofinal in $\mathcal C(G_1 \times G_2)$. 
For $i=1,2$ let $\pi_i:G_1\times G_2 \to G_i$ be the canonical projection. If $F\in \mathcal C(G_1\times G_2)$, then $\pi_i(F)\in \mathcal C(G_i)$ for $i=1,2$, so there is $U_i \in \mathcal N(G_i)$ for $i=1,2$ such that $F\subseteq \pi_1(F)\times \pi_2(F) \subseteq U_1 \times U_2$.
\end{proof}

\section{Algebraic entropy}

In the sequel, for a group $G$ and $\phi\in\End(G)$, if $U$ is a subgroup of $G$, we sometimes use the short notation $T_n$ for $T_n(\phi,U)$, and the equalities $T_n = U \phi ( T_{n-1} )$ and $T_n = T_{n-1} \phi^{n-1} (U)$. In the following lemma we study the latter equality when $U$ commutes with $\phi^n(U)$ for every $n \in \N_+$.
In particular, if $U$ is normal in $G$, we see that every $T_n$ is a subgroup of $G$, normal in $T_{n+1}$. 

\begin{lemma}\label{anna:lemma}
Let $U$ be a subgroup of a group $G$ and $\phi \in \End(G)$. If $\phi^n(U)U=U\phi^n(U)$ for every $n\in\N_+$, then:
\ben
\item for every $k,m\in\N$, $\phi^k(U)\phi^{m}(U)=\phi^m(U)\phi^{k}(U)$;
\item for every $n\in\N_+$, 
\[\phi^n(U)T_n=T_n\phi^n(U),\] 
so $T_{n}$ is a subgroup of $G$.
\een
If $U$ is normal in $G$ (so the above condition is satisfied), then $T_n$ is normal in $T=\bigcup_{k\in \N} T_k$ for every $n\in\N_+$.
\end{lemma}
\begin{proof}
(1) Let $k,m\in\N$ and assume that $k\geq m$. Letting $n=k-m\geq0$ we get $\phi^n(U)U=U\phi^n(U)$ by the hypothesis, hence $\phi^k(U)\phi^m(U)=\phi^m(\phi^n(U)U)=\phi^m(U\phi^n(U))=\phi^m(U)\phi^k(U)$.

(2) Clearly, $T_1 = U \leq G$, and $\phi(U)U=U\phi(U)=T_2$, so $T_2\leq G$.
We proceed by induction on $n\in\N_+$. Let $n\geq2$ and assume the inductive hypothesis $T_{n-1} \phi^{n-1}(U) = \phi^{n-1}(U) T_{n-1}$. Then
\begin{equation*}
\begin{split}
T_n\phi^n(U) = T_{n+1} = U\phi(T_n) = U \phi(T_{n-1}\phi^{n-1}(U)) = U\phi(\phi^{n-1}(U)T_{n-1}) = \\ 
= U\phi^n(U)\phi(T_{n-1}) = \phi^n(U) U \phi(T_{n-1}) = \phi^n(U)T_n.
\end{split}
\end{equation*}
Therefore, $T_{n+1} = T_n\phi^n(U) \leq G$.

Now we assume that $U$ is a normal subgroup of $G$, and we prove that $T_n$ is normal in $T$, checking that  $T_n$ is normal in $T_{n+m}$ for every $m\in \N$. To this aim we fix $m\in \N$, we write $T_{n+m} = T_n\phi^n(T_m)$, and we prove by induction on $n$ that $\phi^n(T_m)$ normalizes $T_n$.

For $n = 1$ this follows from the normality of $T_1 = U$ in $G$. 
Let $n \geq 2$, and fix an element $\phi^n( u ) \in \phi^n(T_m)$. Write $T_n = U \phi(T_{n-1})$, and assume that $\phi^{n-1}( u ) T_{n-1} = T_{n-1} \phi^{n-1}(u)$ holds by the inductive hypothesis. Then
\begin{equation*}\begin{split}
\phi^n(u) T_n = \phi^n(u) U \phi(T_{n-1}) = U \phi^n( u ) \phi(T_{n-1}) = U \phi( \phi^{n-1}( u ) T_{n-1} ) = \\
= U \phi( T_{n-1} \phi^{n-1}( u ) ) = U \phi( T_{n-1} ) \phi^{n}( u ) = T_n \phi^{n}( u ).	\qedhere
\end{split}\end{equation*}
\end{proof}

We now prove that if $U \in \mathcal N(G)$, then one can avoid the use of the right Haar measure in the definition of $H_{alg} (\phi, U)$ given in Equation \eqref{def:H:alg}, and that the limit superior becomes a limit. This generalizes the measure-free formula given in \cite{DG-islam} in the abelian case.

\begin{lemma}\label{lem:conv}
Let $G$ be a locally compact group, $\phi \in \End(G)$ and $U \in \mathcal N(G)$. Then
\begin{equation*}
H_{alg} (\phi, U) = \lim_{n \to \infty} \frac{\log [T_n(\phi,U) :U] }{n}.
\end{equation*}
\end{lemma}
\begin{proof}
If $U \in \mathcal N(G)$, then $T_n = T_n(\phi, U) \in \mathcal B(G)$ by Lemma \ref{anna:lemma}(1), and in particular the index $[T_n:U]$ is finite. Using the fact that $T_n$ is a disjoint union of cosets of $U$, and the properties of $\mu$, we obtain $\mu (T_n) = [T_n:U] \mu (U)$, so that $\log \mu (T_n) = \log [T_n:U] + \log \mu (U)$. As $\log \mu (U)$ does not depend on $n$, passing to the limit superior for $n \to \infty$ we obtain
\[H_{alg} (\phi, U) = \limsup_{n \to \infty} \frac{\log [T_n(\phi,U) :U] }{n}.\]

If $t_n =[T_n:U]$ then $t_n$ divides $t_{n+1}$, as $U \leq T_n \leq T_{n+1}$, and let $\beta_n = t_{n+1}/t_n = [T_{n+1}:T_n]$. Now we show that the sequence of integers $\{ \beta_n \}_{n \geq 1}$ is weakly decreasing. Indeed,
\[\beta_n = [T_{n+1}:T_n] \geq [\phi(T_{n+1}): \phi (T_n)] \geq [U\phi(T_{n+1}): U\phi (T_n)] = [T_{n+2}: T_{n+1}] = \beta_{n+1}.\]
In particular, $\{ \beta_n \}_{n \geq 1}$ stabilizes, so let $n_0 \in \N$, and $\beta \in \N$ be such that for every $n \geq n_0$ we have $\beta_n = \beta$, i.e. $t_n = \beta^{n-n_0} t_{n_0}$. 

Then, \begin{equation}\label{eq:fir}
H_{alg} (\phi, U) = \limsup_{n \to \infty} \frac{\log \beta^{n-n_0} t_{n_0} }{n} = \log \beta = 
\lim_{n \to \infty} \frac{\log t_n }{n}. \qedhere
\end{equation}
\end{proof}

\begin{remark}\label{rem:cofin}
If $G$ is a locally compact group, and $\phi \in \End(G)$, then the assignment $H_{alg} (\phi, - )$ is monotone.  Indeed, if $U \subseteq V$ are elements of $\mathcal C(G)$, then $T_n (\phi, U) \subseteq T_n( \phi, V)$, so $H_{alg}(\phi, U) \leq H_{alg}( \phi, V)$.
So,
\begin{equation*}
h_{alg}(\phi) = \sup \{ H_{alg} (\phi, U) : U \in \mathcal B \}
\end{equation*}
for every cofinal subfamily $\mathcal B$ of $\mathcal C(G)$.
\end{remark} 

If $G$ is a strongly compactly covered group, then by Remark \ref{rem:cofin} we obtain Equation \eqref{eq:algnormal}.
The latter generalizes \cite[Theorem 2.3(b)]{DGB}, which deals with the abelian case.

\medskip
Note that an equivalent formulation of Lemma \ref{cor:cclcstable} is that $\mathcal N(H)$ is cofinal in $\mathcal C(H)$ (respectively, $\mathcal N(G/H)$ is cofinal in $\mathcal C(G/H)$)  when $H$ is a closed (respectively, closed and normal) subgroup of $G$. The following lemma presents smaller cofinal subfamilies of $\mathcal C(H)$ and $\mathcal C(G/H)$ in the same setting.
It has a key role in proving Theorem \ref{addthm} in view of Remark \ref{rem:cofin}. 

\begin{lemma}\label{cofinal:families}
Let $G$ be a strongly compactly covered group, and let $H$ be a closed subgroup of $G$.  Then:
	\ben
	\item  the family $\mathcal N_G(H)=\{U \cap H : U \in \mathcal N(G) \}$ is cofinal in $\mathcal C(H)$;
	\item  if $H$ is also normal in $G$, then the family $\mathcal N_G(G/H)=\{\pi U : U \in \mathcal N(G) \}$ is cofinal in $\mathcal C(G/H)$, where $\pi: G \to G/H$ is the canonical projection.
	\een
\end{lemma}
\begin{proof} 
By Lemma \ref{cor:cclcstable}, it suffices to prove that $\mathcal N_G(H)$ is cofinal in $\mathcal N(H)$,  and $\mathcal N_G(G/H)$ is cofinal in $\mathcal N(G/H)$.
	
Consider the compact open normal subgroups $K$ and $N_g$ of $G$ from Proposition \ref{cclca:structure}.
	
	(1) Let $V \in \mathcal N(H)$. Then $VK \in \mathcal B(G)$, and by our assumption on $G$ there exists $U\in \mathcal N(G)$ containing $VK$.
	Thus,  $U \cap H \geq VK \cap H \geq V \cap H = V$.
	
	(2) Let $V \in \mathcal N(G/H)$. For every $v \in V$, let $u \in \pi^{-1} v$, so that $\langle K,u \rangle\in \mathcal B(G)$ and $\langle K,u \rangle \subseteq N_u\in \mathcal N(G)$, and 
	$v \in \pi (N_u) \in \mathcal N(G/H)$. By the compactness of $V$, there exist $u_1, \ldots, u_n \in G$ such that $V \subseteq \bigcup_{i=1}^n \pi (N_{u_i})$. Then $U =N_{u_1}N_{u_2}\cdots N_{u_n}\in\mathcal N(G)$, and $V\leq\pi U$.
\end{proof}

In the notation of Lemma \ref{cofinal:families}, applying Remark \ref{rem:cofin} we obtain the following result.

\begin{corol}\label{cor:halg:induced:maps}
Let $G$ be a strongly compactly covered group, $\phi \in \End(G)$, and $H$ be a closed $\phi$-invariant subgroup of $G$. Then: 
\ben
\item $\phi \restriction_H \in \End(H)$, and $$h_{alg} (\phi \restriction_H) = \sup \{ H_{alg} (\phi, U) : U \in \mathcal N_G(H) \}.$$
\item if $H$ is also normal, and $\bar \phi : G/H \to G/H$ denotes the induced map, then $\bar \phi \in \End(G/H)$, and 
$$h_{alg} (\bar \phi) = \sup \{ H_{alg} (\phi, U) : U \in \mathcal N_G(G/H) \}.$$
\een
\end{corol}

\section{Basic properties}\label{basic-sec}

In this section we list the basic properties of $h_{alg}$ of endomorphisms of strongly compactly covered groups. We start by showing that the identity map of such groups has zero algebraic entropy.

\begin{lemma}\label{h:alg:id}
If $G$ is a locally compact group, then $H_{alg}(id_G, U)=0$ for every subgroup $U\in\mathcal C(G)$. In particular, if $G$ is strongly compactly covered, then $h_{alg}(id_G)=0$.
\end{lemma}
\begin{proof}
Obviously, if $U\in\mathcal C(G)$ is a subgroup, then $T_n(id_G, U) = U$ for every $n \in \N_+$, so its measure does not depend on $n$, and $H_{alg}(id_G, U) =0$. 

If $G$ is strongly compactly covered, taking the supremum over $U \in \mathcal N(G)$, we obtain $h_{alg}(id_G)=0$ by Equation \eqref{eq:algnormal}.
\end{proof}

Invariance under conjugation was proved in general for endomorphisms of locally compact abelian groups  by Virili \cite[Proposition 2.7(1)]{V}. This property holds true without the abelian assumption as noted in \cite{DG-islam}. For the reader's convenience we provide a proof of this property in Corollary \ref{inv:conj} using the following easy lemma.

\begin{lemma}\label{trajectories:quotient}
	Let $\alpha:G\to G_1$ be a group homomorphism, and   $\phi \in \End(G)$,  $\psi\in \End(G_1)$ be such that $\psi\alpha=\alpha\phi$.
	Then  $T_n(\psi,\alpha K) = \alpha ( T_n(\phi, K) )$ holds for every subset $K$ of $G$, and for every $n \in \N$.
\end{lemma}
\begin{proof} 
	If $K$ is a subset of $G$ and  $n \in \N$, then 
\begin{equation*}\begin{split}
	T_n( \psi,\alpha K) = \alpha  K \cdot \psi \alpha  K \cdots {\psi}^{\, n-1 }\alpha  K &= \alpha  K \cdot \alpha  \phi K \cdots \alpha  \phi^{n-1} K =\\
	& = \alpha ( K \cdot\phi K \cdots\phi^{n-1} K ) = \alpha ( T_n(\phi, K) ).\qedhere
\end{split}\end{equation*}
\end{proof}


\begin{corol}[Invariance under conjugation]\label{inv:conj}
Let $\alpha:G\to G_1$ be  a topological isomorphism of locally compact groups. If $\phi\in \End(G)$ and  $\psi=\alpha \phi\alpha^{-1}$,  then  $H_{alg}(\phi ,K) = H_{alg}(\psi, \alpha K)$  for every $K\in \mathcal C(G)$. In particular, $h_{alg}(\phi)=h_{alg}(\psi).$
\end{corol}
\begin{proof}
Let $\mu$ be a right Haar measure on $G$. For every Borel subset $E\subseteq G_1$  define a right Haar measure $\mu'$ on $G_1$ by letting $\mu'(E)=\mu(\alpha^{-1}(E))$ for every Borel subset $E\leq G_1$. If $K\in\mathcal C(G)$, then $\alpha(K)\in \mathcal C(G_1)$ and  $T_n(\psi,\alpha K) = \alpha ( T_n(\phi, K) )$ by Lemma~\ref{trajectories:quotient}.
It follows that $\mu(T_n(\phi, K))=\mu'(\alpha ( T_n(\phi, K) ) =\mu'(T_n(\psi,\alpha K))$, so $H_{alg} (\phi,  K)=H_{alg} (\psi, \alpha K)$. Since $\alpha^{-1}$ is also an isomorphism, we deduce that $h_{alg}(\phi)=h_{alg}(\psi)$ by definition.
\end{proof}

Now we prove the Logarithmic Law for the algebraic entropy, with respect to endomorphisms of strongly compactly covered groups.

\begin{lemma}[Logarithmic Law] 
Let $G$ be a strongly compactly covered group and $\phi\in \End(G).$ Then  $h_{alg}(\phi^m)=m\cdot h_{alg}(\phi)$ for every $m\in \N.$
\end{lemma}
\begin{proof}
Since the assertion is clear when $m=0$ (see Lemma \ref{h:alg:id}), we may fix $m>0$. Let us show first that $h_{alg}(\phi^m)\leq m\cdot h_{alg}(\phi)$. If $n\in \N_+$ and $U\in \mathcal N(G)$, then
 $U\leq T_n(\phi^m,U)\leq T_{mn-m+1}(\phi, U)$. This implies that  
\begin{equation*}\begin{split}
H_{alg}(\phi^m, U)&=\lim_{n\to \infty}\frac{\log[T_n(\phi^m,U):U]}{n}\leq
\\ &\leq \lim_{n\to \infty} \frac{\log[T_{mn-m+1}(\phi, U):U]}{mn-m+1}\cdot \lim_{n\to \infty}\frac{mn-m+1}{n} = mH_{alg}(\phi, U),
\end{split}\end{equation*}
and taking the suprema over $U\in \mathcal N(G)$ we obtain $h_{alg}(\phi^m)\leq m\cdot h_{alg}(\phi)$.
 
To prove the converse inequality, let $U\in \mathcal N(G)$. Observe that $$W = T_n(\phi^m, T_m(\phi,U))=T_{nm}(\phi,U),$$ so both $T_m(\phi,U)\in \mathcal B(G)$ and $W\in \mathcal B(G)$ by Lemma \ref{anna:lemma}(1), and $U \leq T_m(\phi,U) \leq W$. Then
\begin{equation*}\begin{split}
h_{alg}(\phi^m)\geq H_ {alg} (\phi^m, T_m(\phi,U)) &= \limsup_{n\to \infty}\frac{\log \mu(T_n ( \phi^m, T_m(\phi,U) ))}{n} = \\
&= \limsup_{n\to \infty}\frac{\log \mu(T_{nm} ( \phi, U ))}{n} = m H_ {alg} (\phi,U).
\end{split}\end{equation*}
As the above inequality holds for every $U\in \mathcal N(G)$, this proves that $h_{alg}(\phi^m)\geq m\cdot h_{alg}(\phi)$, and thus $h_{alg}(\phi^m)= m\cdot h_{alg}(\phi)$.
\end{proof}

Given two group endomorphisms $\phi_i : G_i \to G_i$ for $i = 1,2$, in the following proposition we consider the group endomorphism $\phi_1\times \phi_2 : G_1 \times G_2 \to G_1 \times G_2$, mapping $(x,y) \mapsto ( \phi_1 (x), \phi_2 (y) )$.

\begin{prop}[weak Addition Theorem] 
Let $G_1,G_2$ be strongly compactly covered groups, and let $\phi_1\in \End(G_1)$, $\phi_2\in \End(G_2)$. 
Then $h_{alg}(\phi_1\times \phi_2)=h_{alg}(\phi_1)+h_{alg}(\phi_2)$.
\end{prop}
\begin{proof}
By Lemma \ref{lem:cofprod}, the family $\{U_1\times U_2: U_i\in \mathcal N(G_i)\}$ is cofinal in $\mathcal C (G_1\times G_2)$, and $G_1\times G_2$ is a strongly compactly covered group.

For $i=1,2$ let $U_i\in \mathcal N(G_i)$, and $n \geq 1$ be an integer. Then $$W = T_n ( \phi_1\times \phi_2, U_1\times U_2 ) = T_n ( \phi_1, U_1 ) \times T_n ( \phi_2, U_2 ),$$ so 
\[ [W : U_1\times U_2 ] = [ T_n ( \phi_1, U_1 ) : U_1 ] \cdot [ T_n ( \phi_2, U_2 ) : U_2].\]
Applying $\log$, dividing by $n$, and taking the limit for $n \rightarrow \infty$ we conclude that
\begin{equation}\label{weak:add:thm:formula}
H_{alg} ( \phi_1\times \phi_2, U_1\times U_2 ) =H_{alg} ( \phi_1, U_1 ) + H_{alg} ( \phi_2, U_2 ).
\end{equation}

Now we take the suprema over $U_1\in \mathcal N(G_1)$ and $U_2\in \mathcal N(G_2)$ in Equation \eqref{weak:add:thm:formula}. Then the left hand side gives $h_{alg}(\phi_1\times \phi_2)$ by Lemma \ref{lem:cofprod}. Taking into account that $H_{alg}(\phi_1,U_1)$ only depends on $U_1$, while $H_{alg}(\phi_1,U_2)$ only depends on $U_2$, the right hand side of Equation \eqref{weak:add:thm:formula} gives 
\begin{multline*}
\sup\{H_{alg}(\phi_1,U_1)+H_{alg}(\phi_1,U_2): \ U_i\in \mathcal N(G_i), \ i=1,2 \} =
\\= \sup\{H_{alg}(\phi_1,U_1) : \ U_1\in \mathcal N(G_1)\} +
 \sup\{H_{alg}(\phi_2,U_2): \ U_2\in \mathcal N(G_2)\} = 
h_{alg}(\phi_1)+h_{alg}(\phi_2).		\qedhere
\end{multline*}
\end{proof}

We briefly mention here the monotonicity property of $h_{alg}$ for endomorphisms of strongly compactly covered groups, that will immediately follow from the stronger result proved in Proposition \ref{first:half:add:thm}.

\begin{prop}[Monotonicity]\label{monotonicity}
Let $G$ be a strongly compactly covered group, $\phi \in \End(G)$, $H$ a closed normal $\phi$-invariant subgroup of $G$, and $\bar \phi : G/H \to G/H$ the induced map. Then 
\begin{equation*}
h_{alg} (\phi) \geq \max\{h_{alg} (\bar \phi),h_{alg} (\phi \restriction_H)\}.
\end{equation*}
\end{prop}

Let $F$ be a finite discrete abelian group and let $G=G_c\oplus G_d$, with
\[G_c=\prod_{n=-\infty}^0 F\quad\mbox{and}\quad G_d=\bigoplus_{n=1}^\infty F.\]
Endow $G_d$ with the discrete topology and $G_c$ with the (compact) product topology, so both of them are strongly compactly covered. 
Endowed with the product topology, $G$ is locally compact, and it is also strongly compactly covered by Lemma \ref{lem:cofprod}.

The \emph{left shift} is
$${}_F\beta\colon G\to G, \quad (x_n)_{n\in\Z}\mapsto (x_{n+1})_{n\in\Z},$$
while the \emph{right shift} is
$$\beta_F\colon G\to G, \quad (x_n)_{n\in\Z}\mapsto(x_{n-1})_{n\in\Z}.$$
Clearly, $\beta_F$ and ${}_F\beta$ are topological automorphisms of $G$ such that ${}_F\beta^{-1}=\beta_F.$

\smallskip
In the next example we compute the algebraic entropy of these shifts.

\begin{example}\label{ex:bs}\label{bernoulli}
Let $p$ be a prime number, and $\Z_p$ be the cyclic group with $p$ elements.
Consider $F=\Z_p$, i.e., $G= \prod_{n=-\infty}^0 \Z_p\oplus \bigoplus_{n=1}^\infty \Z_p$. Since $G$ is abelian, $\mathcal N(G)=\mathcal B(G)$.
We verify now that
\begin{equation*}
h_{alg}({}_{\Z_p}\beta)=0\quad\mbox{and}\quad h_{alg}(\beta_{\Z_p})=\log p.
\end{equation*}

For $k\in\N_+$, let 
\[U_k=\prod_{n=-\infty}^0\Z_p\times\bigoplus_{n=1}^k \Z_p\in\mathcal B(G).\]
Let us check that the family $\mathcal U=\{U_k:k\in\N_+\}$ is cofinal in $\mathcal C(G)$. Indeed, let $H \in \mathcal C(G)$, and consider the canonical projection $\pi_d : G \to G_d$. Then $\pi_d(H) \subseteq G_d$ is finite, so $\pi_d(H) \subseteq \bigoplus_{n=1}^k \Z_p$ for some $k \in \N_+$, and $H \subseteq U_k$.

Clearly,
\begin{equation*}
\ldots\leq {}_{\Z_p}\beta^n(U_k)\leq\ldots\leq{}_{\Z_p}\beta(U_k)\leq U_k\leq\beta_{\Z_p}(U_k)\leq\ldots\leq\beta_{\Z_p}^n(U_k)\leq\ldots.
\end{equation*}

For all $n\in\N_+$, $T_n({}_{\Z_p}\beta,U_k)=U_k$, so $$H_{alg}({}_{\Z_p}\beta,U_k)=0.$$
On the other hand, for every $n\in\N_+$,
$$\log\frac{T_{n+1}(\beta_{\Z_p},U_k)}{T_n(\beta_{\Z_p},U_k)}=\log\frac{\beta^{n}_{\Z_p}(U_k)}{\beta^{n-1}_{\Z_p} (U_k)}=\log \frac{\beta_{\Z_p}(U_k)}{U_k}=\log p,$$
hence $$H_{alg}(\beta_{\Z_p},U_k)=\log p.$$ 
By the cofinality of $\mathcal U$ in $\mathcal C(G)$ and by Remark \ref{rem:cofin} we can conclude.
\end{example}

The above example shows in particular that a topological automorphism $\phi$ of a strongly compactly covered group and its inverse $\phi^{-1}$ can have different algebraic entropy. 
Next we give the precise relation between $h_{alg}(\phi)$ and $h_{alg}(\phi^{-1})$, which extends \cite[Proposition 2.7(3)]{V} to some non-abelian groups.

\smallskip
For a locally compact group $G$, let $\Aut(G)$ denote the group of topological automorphisms of $G$. If $\mu$ is a left Haar measure on $G$, the \emph{modulus} is a group homomorphism $\Delta_G:\Aut(G)\to\R_+$ such that $\mu(\phi E)=\Delta_G(\phi)\mu(E)$ for every $\phi\in \Aut(G)$ and every measurable subset $E$ of $G$ (see \cite{HR}). If $G$ is either compact or discrete, then $\Delta_G(\phi)=1$ for every $\phi\in \Aut(G)$. We also denote $\Delta_G$ simply by $\Delta$.

\begin{lemma}\label{Delta}
Let $G$ be a locally compact group, $\phi\in \Aut(G)$, and $U\in\mathcal N(G)$. Then
\ben
\item  $\Delta(\phi)=\frac{\mu(\phi(U))}{\mu(U)}=\frac{[U\phi(U):U]}{[U\phi(U):\phi(U)]}$.
\item If $V$ is a compact subgroup of $G$ and $V\supseteq U\phi(U)$, then $[V:U]=[V:\phi(U)]\cdot\Delta(\phi)$.
\een
\end{lemma}
\begin{proof}
(1) Since $U$ and $\phi(U)$ are compact, and $U\cap\phi(U)$ is open, it follows that $[\phi(U):U\cap\phi(U)]$ and $[U:U\cap\phi(U)]$ are finite. Therefore, $\mu(\phi(U))=[\phi(U):U\cap \phi(U)]\cdot\mu(U\cap\phi(U))$ and $\mu(U)=[U:U\cap\phi(U)]\cdot \mu(U\cap\phi(U))$. Hence, $$\Delta(\phi)=\frac{\mu(\phi(U))}{\mu(U)}=\frac{[\phi(U):U\cap\phi(U)]}{[U:U\cap\phi(U)]}.$$
Moreover, since $U$ is normal in $G$ we have that $U\phi(U)$ is a subgroup of $G$, and so $[\phi(U):U\cap\phi(U)]=[U\phi(U):U]$ and $[U:U\cap\phi(U)]=[U\phi(U):\phi(U)]$.

(2) We have that $$[V:U]=[V:U\phi(U)][U\phi(U):U]$$ and $$[V:\phi(U)]=[V:U\phi(U)][U\phi(U):\phi(U)].$$ Therefore, $$[V:U]=[V:\phi(U)]\frac{[U\phi(U):U]}{[U\phi(U):\phi(U)]},$$
and the conclusion follows from item (1).
\end{proof}

\begin{prop}\label{h:alg:phi:inverse}
Let $G$ be a locally compact group, $\phi\in \Aut(G)$ and $U\in\mathcal N(G)$. Then
\[H_{alg}(\phi^{-1},U)= H_{alg}(\phi,U)-\log\Delta(\phi).\]
In particular, if $G$ is a strongly compactly covered group, then 
\[h_{alg}(\phi^{-1})=h_{alg}(\phi)-\log\Delta(\phi).\]
\end{prop}
\begin{proof}
Let $U\in\mathcal N(G)$.
As in the proof of Lemma \ref{lem:conv}, for every $n\in\N_+$ let $t_n=[T_n(\phi,U):U]$ and $t_n^*=[T_n(\phi^{-1},U):U]$.
Moreover, $H_{alg}(\phi,U)=\log\beta$ and $H_{alg}(\phi^{-1},U)=\log \beta^*$, where $\beta$ and $\beta^*$ are respectively the values at which the sequences $\beta_n=\frac{t_{n+1}}{t_{n}}$ and $\beta^*_n=\frac{t_{n+1}^*}{t_{n}^*}$ stabilize (see Equation \eqref{eq:fir}). 

Let $n\in\N_+$. By Lemma \ref{anna:lemma} we have that
\[\phi^{n-1}(T_n(\phi^{-1},U))=T_n(\phi,U)\] 
is a compact subgroup of $G$. Since $\phi^{n-1}$ is an automorphism, Lemma \ref{Delta}(2) gives
\begin{equation*}\begin{split}
t_n^* &=[T_n(\phi^{-1},U):U]=[\phi^{n-1}(T_n(\phi^{-1},U)):\phi^{n-1}(U)]=\\&=[T_n(\phi,U):\phi^{n-1}(U)]=[T_n(\phi,U):U]\cdot\frac{1}{\Delta(\phi^{n-1})}=t_n\cdot\frac{1}{\Delta(\phi^{n-1})}.
\end{split}\end{equation*}
Therefore, since $\Delta$ is a homomorphism, for every sufficiently large $n\in\N_+$, $$\beta=\frac{t_{n+1}}{t_n}=\frac{t_{n+1}^*}{t_n^*}\cdot\frac{\Delta(\phi^{n})}{\Delta(\phi^{n-1})}=\beta^*\cdot\Delta(\phi).$$ Then, $\log\beta=\log\beta^*+\log\Delta(\phi)$, that is, the first assertion of the proposition.

The second equality in the thesis follows from the first one taking the supremum for $U\in\mathcal N(G)$ in view of Equation \eqref{eq:algnormal}.
\end{proof}

\section{Limit-free Formula}\label{sec:LF}

In this section we prove the so-called Limit-free Formula for the algebraic entropy. We start  introducing the following useful subgroups.

  
\begin{defi}\label{def:umin}
For a locally compact group $G$, $\phi\in \End(G)$ and $ U\in \mathcal N(G)$, let:
\ben \item $U^{(0)}=U;$  
\item $U^{(n+1)}=U\phi^{-1}U^{(n)}$ for every $n\in \N;$
\item $U^-=\bigcup_{n\in\N} U^{(n)}.$
\een
\end{defi}
Is is easy to prove by induction that $U^{(n+1)} = \phi^{-1}U^{(n)} U$ is an open normal subgroup of $G$ such that $U^{(n)} \leq U^{(n+1)}$ for every $n \in \N$, so that also $U^-$ is an open normal subgroup of $G$ for every $U\in \mathcal N(G).$
We collect some of the properties of $U^-.$  

\begin{lemma}\label{lem:pro}
Let $G$ be a locally compact group, $\phi\in \End(G)$ and $ U\in \mathcal N(G).$ Then:
\ben
\item $\phi^{-1}U^{-} \leq U^-$; 
\item if $H\leq G$ is such that $U\leq H$ and $\phi^{-1}H\leq H,$ then $U^{-}\leq H$;
\item $U^- = U\phi^{-1}U^-$;
\item the index $[U^- :\phi^{-1}U^-] = [U:U\cap \phi^{-1}U^-]$ is finite;
\item for every $n\in \N^{+}, \phi^{-n}T_n = \phi^{-1}U^{(n-1)}. $
\een
\end{lemma}
\begin{proof}
(1) Follows from the fact that $\phi^{-1}U^{(n)}\leq U^{(n+1)}$ for every $n\in \N.$

(2) It suffices to show that $U^{(n)}\leq H$ for every $n\in \N$. If $n=0$, then $U=U^{(0)}\leq H$, while for the inductive step, use $\phi^{-1}U^{(n)}\leq \phi^{-1} H \leq H$.

(3) We have already noted above that $U^{(n)}\leq U^{(n+1)}$ for every $n\in N$, so
\begin{equation*}
U \phi^{-1}U^- = U\phi^{-1}\bigcup_{n\in\N} U^{(n)} = U\bigcup_{n\in\N}\phi^{-1}U^{(n)} = 
\bigcup_{n\in\N}(U\phi^{-1}U^{(n)})  = \bigcup_{n\in\N}U^{(n+1)} 
= U^-.
\end{equation*}

(4) Using property (3) we obtain
\begin{equation*}
U^-/\phi^{-1}U^- =(U \phi^{-1}U^-)/\phi^{-1}U^-\cong U/(U\cap \phi^{-1}U^-),
\end{equation*}
and thus $[U^- :\phi^{-1}U^-]=  [U:U\cap \phi^{-1}U^-].$ The quotient $U/(U\cap \phi^{-1}U^-)$ is finite, since $U\cap \phi^{-1}U^-$ is an open subgroup of the compact group $U$, as $\phi^{-1}U^-$ is open in $G$.

(5) For $n=1$ the assertion follows from the equalities $T_1=U=U^{(0)}.$  We prove the assertion for $n+1$ assuming  that it holds true for $n\in \N_+.$

To see first that $\phi^{-n-1}T_{n+1} \subseteq \phi^{-1}U^{(n)}$, fix $x\in \phi^{-n-1}T_{n+1}.$ So, $\phi^{n+1}(x)\in T_{n+1}=T_n\phi^nU.$ This implies that there exists  $y\in U$ such that $\phi^{n+1}(x)\in T_n\phi^n(y)$, that is,
$\phi^n(\phi(x)y^{-1})\in T_n.$ By the inductive hypothesis,  $\phi(x)y^{-1}\in \phi^{-n}T_n=\phi^{-1}U^{(n-1)}.$
Hence, $$\phi(x)=(\phi(x)y^{-1})y\in \phi^{-1}(U^{(n-1)})U = U \phi^{-1}(U^{(n-1)}) = U^{(n)}.$$ This means that $x\in \phi^{-1}U^{(n)},$ as needed. 

For the converse inclusion let $x\in \phi^{-1}U^{(n)}.$ By the inductive hypothesis, $$\phi(x)\in U^{(n)}=U\phi^{-1}U^{(n-1)}=U\phi^{-n}T_n.$$ It follows that $$\phi^{n+1}(x)=\phi^{n}(\phi(x))\in \phi^n(U)T_n=T_{n+1},$$ and thus, $x\in \phi^{-n-1}T_{n+1}$. 
\end{proof}

In the following result, we generalize the Limit-free Formula for the algebraic entropy of endomorphism of torsion abelian groups, from \cite{DG-lf,Y}.

\begin{prop}[Limit-free Formula] \label{prop:limit-free}
Let $G$ be a locally compact group. If $\phi\in \End(G)$ and $ U\in \mathcal N(G)$, then $$H_{alg} (\phi,U)= \log[U^- :\phi^{-1}U^-].$$
\end{prop}
\begin{proof}
By Equation \eqref{eq:fir},  $H_{alg} (\phi, U)=\log \beta,$ where $\beta=\beta_n = [T_{n+1}:T_n]$ for $n\geq n_0$ for some $n_0\in \N.$ So, it suffices to show that $\beta=[U^- :\phi^{-1}U^-].$ 

For every $n\in \N$, 
\[U\cap \phi^{-1}U= U\cap \phi^{-1}U^{(0)}\leq U\cap \phi^{-1}U^{(n)}\leq U\cap \phi^{-1}U^{(n+1)} \leq U,\] 
so the sequence of the indices $\{[U: U\cap \phi^{-1}U^{(n)}]\}_{n\in\N}$ is weakly decreasing, hence it stabilizes.
It follows that the sequence $\{U\cap \phi^{-1}U^{(n)}\}_{n\in\N}$ of subgroups of $U$ also stabilizes, and let $n_1\in \N$ be such that $U\cap \phi^{-1}U^{(n)}=U\cap \phi^{-1}U^{(n_1)}$ for every $n\geq n_1$. 

For $m\geq n_1$ we have  $$U\cap \phi^{-1}U^{(m)}=\bigcup_{n\in\N}(U\cap \phi^{-1}U^{(n)})=U\cap \bigcup_{n\in\N}\phi^{-1}U^{(n)}=U\cap \phi^{-1}\bigcup_{n\in\N}U^{(n)}= U\cap \phi^{-1}U^-.$$   Therefore, for every $n\geq \max\{n_0,n_1\},$  Lemma \ref{lem:pro}(4)  gives 
\begin{equation*}\begin{split}
[U^- :\phi^{-1}U^-]=&  [U:U\cap \phi^{-1}U^-]=[U:U\cap \phi^{-1}U^{(n)}]=\\
&=[U\phi^{-1}U^{(n)}:\phi^{-1}U^{(n)}]=[U^{(n+1)}:\phi^{-1}U^{(n)}].
\end{split}\end{equation*}

To conclude, we show that $[U^{(n)}:\phi^{-1}U^{(n-1)}]=[T_{n+1}:T_n]=\beta$. 
Observe that 
\[U^{(n)}/\phi^{-1}U^{(n-1)}=(U\phi^{-1}U^{(n-1)})/\phi^{-1}U^{(n-1)},\] and that $T_{n+1} /T_n=(\phi^{n}(U)T_n)/T_n$. 
Consider the mapping 
\[\Phi:(U\phi^{-1}U^{(n-1)})/\phi^{-1}U^{(n-1)}\to (\phi^{n}(U)T_n)/T_n, \ x\phi^{-1}U^{(n-1)}\mapsto \phi^{n}(x)T_n.\]
By Lemma \ref{lem:pro}(5), $\Phi$ is well-defined and injective, and clearly $\Phi$ is a surjective homomorphism.
\end{proof}

The following result was inspired by Fact \ref{LFcortop} below (proved in \cite{GBV}) which is a similar result for the topological entropy.

\begin{prop}\label{LFcoralg}
Let $G$ be a strongly compactly covered group, and $\phi\in \End(G)$. Then
$$h_{alg}(\phi)=\sup\{\log[A:\phi^{-1}A]: A\unlhd G,\ A\ \text{open},\  \phi^{-1}A\leq A,\ [A:\phi^{-1}A]<\infty\}=:s.$$
\end{prop}
\begin{proof}
Using Proposition \ref{prop:limit-free} we obtain 
\[h_{alg}(\phi)=\sup\{H_{alg}(\phi, U): U\in \mathcal N(G)\}=\sup\{\log[U^-:\phi^{-1}U^-]: U\in \mathcal N(G)\}.\]
Then  $h_{alg}(\phi)\leq s$, as for every $U\in \mathcal N(G)$, $U^-$ is an open normal subgroup of $G$ such that $\phi^{-1}U^-\leq U^-$ and $[U^-:\phi^{-1}U^-]<\infty$ by Lemma \ref{lem:pro}(4).  

For the converse inequality, fix an open normal subgroup $A$ of $G$ such that $\phi^{-1}A\leq A$ and $[A:\phi^{-1}A]<\infty.$ Our aim is to find $U\in \mathcal N(G)$ such that $[U^-:\phi^{-1}U^-]\geq [A:\phi^{-1}A]$. Since $[A:\phi^{-1}A]<\infty$, there exists a finitely generated subgroup $F\leq A$ such that $A=\phi^{-1}(A)F$. As $G$ is compactly covered by Proposition \ref{cclca:structure}, there exists a compact subgroup $K\leq G$ such that $F\leq K.$ Clearly, $A\cap K\in \mathcal N(K).$ So, Lemma \ref{cofinal:families}(1) implies that there exists $M\in \mathcal N(G)$ such that $A\cap K\subseteq M\cap K.$ We claim that 
\begin{equation} \label{eq:coadd}
A=\phi^{-1}(A)U, 
\end{equation} 
where $U=M\cap A\in \mathcal N(G).$ First observe that  
$$F\subseteq A\cap K = M\cap K\cap A\subseteq M\cap A=U\subseteq A.$$ 
Taking also into account that $\phi^{-1}A\leq A$ we obtain
\[A=\phi^{-1}(A)F\subseteq \phi^{-1}(A)U\subseteq \phi^{-1}(A)A=A,\] which proves Equation \eqref{eq:coadd}.

We now show that $[U^-:\phi^{-1}U^-]\geq [A:\phi^{-1}A].$
Since $U\leq A$ and $\phi^{-1}A\leq A$, Lemma \ref{lem:pro}(2) gives $U^-\leq A$, so  $\phi^{-1}U^-\leq \phi^{-1}A\leq A$. 
Now  by Equation \eqref{eq:coadd}, 
$$[U\phi^{-1}U^-:\phi^{-1}U^-]\geq [(U\phi^{-1}(U^-))\phi^{-1}A:\phi^{-1}(U^-)\phi^{-1}A]=[A:\phi^{-1}A].$$
Finally, $[U^-:\phi^{-1}U^-]=[U\phi^{-1}U^-:\phi^{-1}U^-]$ by Lemma \ref{lem:pro}(3).
\end{proof}

\section{Bridge Theorem}\label{BT-sec}

For a locally compact abelian group $G$, its Pontryagin dual $\widehat G$ is the group of all continuous homomorphism $G \to \T$, equipped with the pointwise operation, and the compact open topology; by Pontryagin duality, $\widehat G$ is a topological group such that $\widehat {\widehat G\, }$ is canonically topologically isomorphic to $G$. 

For a subgroup $H$ of $G$, the annihilator of $H$ in $\widehat G$ is 
$H^\perp = \{ \chi \in \widehat G : \chi(H) = 0 \}$,
while for a subgroup $A$ of $\widehat G$, the annihilator of $A$ in $G$ is $A^\top = \{ x \in G : \chi(x) = 0 \text{ for every } \chi \in A \}$. 

Then, for every closed subgroup $H$ of $G$, we have $(H^\perp )^\top = H$, while $H\leq G$ is compact if and only if $H^\perp \leq \widehat G$ is open. In particular, $H \in \mathcal B(G)$ if and only if $H^\perp \in \mathcal B( \widehat G)$.

Finally, if $\{ H_i \}_{i \in I}$ is a family of open subgroups of $G$, then
\begin{equation*}
\Big( \sum_{i \in I} H_i \Big)^\perp = \bigcap_{i \in I} H_i^\perp 
\text{\ \ and \ \ }
\Big( \bigcap_{i \in I} H_i \Big)^\perp = \sum_{i \in I} H_i^\perp.
\end{equation*}

In what follows, we need also the following standard properties of Pontryagin duality.

\begin{lemma}\label{dual:of:quotient}
Let $B\leq A$ be closed subgroups of a locally compact abelian group $G$. Then $\widehat{ A/B}$ is topologically isomorphic to $B^\perp/A^\perp$.
\end{lemma}

For a continuous group homomorphism $\phi : G \to H$, its dual is the continuous group homomorphism $\hat \phi : \widehat H \to \widehat G$, mapping $f \mapsto f \circ \phi$.

\begin{lemma}\label{annihilator:of:preimage}
Let $G$ be a locally compact abelian group, $\phi\in \End(G)$ and $U$ be a closed subgroup of $G$. Then $( \phi^{-1}U )^\perp = \widehat \phi( U^\perp )$.
\end{lemma}

\begin{remark}\label{cc-td}
For a topological abelian group $G$, the subgroup 
$$B(G) = \sum \{ K\leq G : K \text{ compact} \}$$
is the biggest compactly covered subgroup of $G$, so $G$ is compactly covered exactly when $B(G) = G$.
When $G$ is locally compact abelian, as $B(G)^\perp = c( \widehat G)$, one has that $G$ is compactly covered if and only if $\widehat G$ is totally disconnected.
\end{remark}

The following result is a consequence of the Limit-free Formula for the topological entropy proved in \cite[Proposition 3.9]{GBV}.
\begin{fact}\cite[Equation (3.10)]{GBV}\label{LFcortop}
Let $G$ be a totally disconnected locally compact group and $\phi\in\End(G)$. Then
$$h_{top}(\phi)=\sup\{\log[\phi M:M]: M\leq G,\ M\ \text{compact},\ M\subseteq\phi M,\ [\phi M:M]<\infty\}.$$
\end{fact}

Using Proposition~\ref{LFcoralg} and Fact~\ref{LFcortop}, which are respectively consequences of the Limit-free Formulas for the algebraic and for the topological entropy, we give in Theorem \ref{BT} a new and very short proof for the Bridge Theorem from \cite{DGB}.
 
\begin{thm}\label{BT}
Let $G$ be a compactly covered locally compact abelian group and $\phi\in\End(G)$. Then $h_{alg}(\phi)=h_{top}(\widehat\phi)$.
\end{thm}
\begin{proof}
Observe that $\widehat G$ is totally disconnected by Remark \ref{cc-td}.

Let $A$ be a subgroup of $G$. Then the following conditions are equivalent:
\begin{enumerate}[(i)]
\item $A$ is open, $\phi^{-1}A\subseteq A$ and $A/\phi^{-1}A$ is finite;
\item $M=A^\perp$ is compact, $M\subseteq\widehat\phi M$ and $\widehat\phi M/M$ is finite.
\end{enumerate}
Under these conditions, by Lemma \ref{annihilator:of:preimage} and Lemma \ref{dual:of:quotient} respectively, 
\[\widehat \phi M/M =  ( \phi^{-1} A )^\perp / A ^\perp \cong \widehat{ A/\phi^{-1}A },\] 
which is isomorphic to $A/\phi^{-1}A$, being finite.
Now the thesis follows from Fact~\ref{LFcortop} and Proposition~\ref{LFcoralg}.
\end{proof}

\section{Addition Theorem}\label{sec:AT}

This section is dedicated to the following Addition Theorem.

\begin{thm}\label{addthm}
Let $G$ be a strongly compactly covered group, $\phi\in \End(G)$, $H$ a closed normal $\phi$-stable subgroup of $G$ with $\ker\phi\leq H$, and $\bar \phi : G/H \to G/H$ the induced map. Then 
\begin{equation*}
h_{alg} (\phi) =h_{alg} (\bar \phi) + h_{alg} (\phi \restriction_H).
\end{equation*}
\end{thm}

As an immediate corollary we obtain the following result for a topological automorphism $\phi\in \Aut(G)$ of $G$.

\begin{corol}\label{ATaut}
If $G$ is a strongly compactly covered group, $\phi\in \Aut(G)$ and $H$ is a normal $\phi$-stable subgroup of $G$, then  \begin{equation*}
h_{alg} (\phi) =h_{alg} (\bar \phi) + h_{alg} (\phi \restriction_H).
\end{equation*}
\end{corol}
In particular, the Addition Theorem holds for automorphisms of discrete torsion FC-groups. These groups are locally finite, and we conjecture the Addition Theorem holds also in the bigger class of discrete locally finite groups.

\begin{conj}\label{conj}
	If $G$ is a discrete locally finite group, $\phi\in \Aut(G)$ and $H$ is a normal $\phi$-stable subgroup of $G$, then  \begin{equation*}
	h_{alg} (\phi) =h_{alg} (\bar \phi) + h_{alg} (\phi \restriction_H).
	\end{equation*}
\end{conj}

In the sequel, $G$ is a strongly compactly covered group, $\phi \in \End(G)$, and $H$ is a closed normal $\phi$-invariant subgroup of $G$. We denote by $\pi: G \to G/H$ the canonical projection, and by $\bar \phi : G/H \to G/H$ the map induced by $\phi$ on the quotient group.
Obviously, $\bar \phi \pi = \pi \phi$.

First we prove the following easy charaterization of the subgroups $H$ of $G$  that we consider.
\begin{lemma}\label{lem:stable}
Let $G$ be a group, $\phi\in \End(G)$ and $H$ a closed  subgroup of $G$. The following assertions are equivalent:
\ben 
\item  $H$ is   $\phi$-stable and contains $\ker(\phi)$;
\item  $\phi^{-1}(H)=H$ and $H\leq \mathrm{Im}(\phi)$.
\een
\end{lemma}
\begin{proof}
(1)$\Rightarrow$(2) Since $H$ is   $\phi$-stable,  we have $H=\phi(H)\leq \phi(G)=\mathrm{Im}(\phi)$ and  $\phi^{-1}(H)=\phi^{-1}(\phi(H))$. As $\ker(\phi)\leq H$,  $\phi^{-1}(H)=\phi^{-1}(\phi(H))=H$.

(2)$\Rightarrow$(1) Since $\ker(\phi)\leq \phi^{-1}(H)$, the condition $\phi^{-1}(H)=H$ implies that $\ker(\phi)\leq H$ and also that $\phi(\phi^{-1}(H))=\phi(H)$. As $H\leq \mathrm{Im}(\phi)$,  $\phi(\phi^{-1}(H))$ coincides also with $H$. It follows that $H=\phi(H)$.
\end{proof}

Now we prove the first half of the Addition Theorem. Note that it holds in a more general setting than Theorem \ref{addthm}.

\begin{prop}\label{first:half:add:thm}
Let $G$ be a strongly compactly covered group, $\phi \in \End(G)$, $H$ a closed normal $\phi$-invariant subgroup of $G$, and $\bar \phi : G/H \to G/H$ the induced map. Then 
\begin{equation*}
h_{alg} (\phi) \geq h_{alg} (\bar \phi) + h_{alg} (\phi \restriction_H).
\end{equation*}
\end{prop}
\begin{proof}
For $U\in \mathcal N(G)$, and $n \in \N$, let $T_n = T_n(\phi,U)$. Then $U \leq T_n \cap (UH) \leq T_n$, so that 
\begin{equation}\label{two:indexes}
[T_n : U] = [T_n: T_n \cap (UH)] [T_n \cap (UH):U],
\end{equation} 
and we study separately the two indices in the right hand side of the above equation. 

Let $a = [T_n: T_n \cap (UH)] $ and consider the following Hasse diagram in the lattices of subgroups of $G$ and of $G/H$.
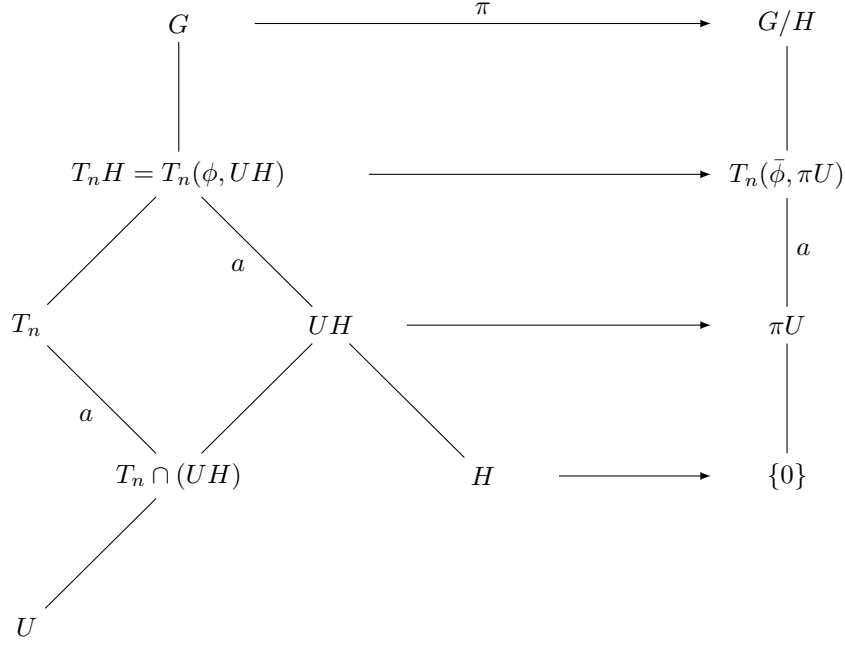
\begin{center}
\begin{tikzpicture}
\node[above] at (4,0) {$\pi$};
\node(G) at (0,0) {$G$};											\node(G/H) at (8,0) {$G/H$};
\node (TnH) at (0,-2) {$T_n  H = T_n(\phi, UH)$};			\node(piTn) at (8,-2) {$T_n(\bar \phi, \pi U)$};
\node[below left] at (1,-3) {$a$};								\node [right] at(8,-3) {$a$};
\node(UH)      
at (2,-4) {$UH$};										\node(piU) at (8,-4) {$\pi U$};
\node(Tn)      
at (-2,-4) {$T_n$};
\node[below left] at (-1,-5) {$a$};
\node(inters)     
at (0,-6)  {$T_n \cap (UH)$};
\node(U)     
at (-2,-8)     {$U$};
\node(H)    
at (4,-6)  {$H$};										\node(o) at (8,-6) {$\{0\}$};
\draw(G) -- (TnH);
\draw(TnH)       -- (UH);
\draw(TnH)       -- (Tn);
\draw(Tn)       -- (inters);
\draw(UH)       -- (inters);
\draw(inters)       -- (U);
\draw(UH)       -- (H);
\draw(G/H) -- (piTn) -- (piU) -- (o);
\draw[-latex](1,0)--(7,0);
\draw[-latex](2.5,-2)--(7,-2);
\draw[-latex](3,-4)--(7,-4);
\draw[-latex](5,-6)--(7,-6);
\end{tikzpicture}
\end{center}

Note that $T_n H = T_n(\phi, UH)$, since $\phi H\leq H$ and both $U$ and $H$ are normal in $G$, so
\[a= [T_n  H : UH] = [ T_n(\phi, UH) : UH].\]
Moreover, both $T_n(\phi, UH)$ and $UH$ contain $H = \ker \pi$, so considering their images in $G/H$ we have
$[ T_n(\phi, UH) : UH]= [\pi ( T_n(\phi, UH) ) : \pi(UH)]$, and the latter coincides with $[T_n(\bar \phi, \pi U ):  \pi U]$ by Lemma \ref{trajectories:quotient}.

To study the second index in the right hand side of Equation \eqref{two:indexes}, let $b = [T_n \cap (UH):U]$. As $T_n \cap (UH) = (T_n\cap H) U$ by the modular law, we have 
\[b = [(T_n\cap H) U : U] = [T_n \cap H:U\cap H] \geq [T_n(\phi, U\cap H):U\cap H],\]
where the last inequality follows from the inclusions $T_n \cap H \geq T_n(\phi, U\cap H) \geq U \cap H$.
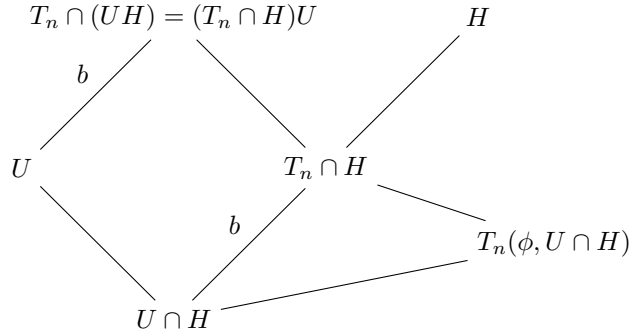
\begin{center}
\begin{tikzpicture}
\node(TncapH U) at(0,0)     
{$
T_n \cap (UH) = (T_n\cap H)U$};		\node[above left] at (-1,-1) {$b$};
\node(U)     
at (-2,-2)     {$U$};
\node(Tn intH)     
at (2,-2)     {$T_n \cap H$};
\node(H)    
at (4,0)    {$H$};
\node (Tnint) at (5,-3) {$T_n(\phi, U\cap H)$};
\node(UintH)    
at (0,-4)    {$U\cap H$};				\node[above left] at (1,-3) {$b$};
\draw(TncapH U)       -- (U);
\draw(TncapH U)       -- (Tn intH);
\draw(H)       -- (Tn intH);
\draw(U)       -- (UintH);
\draw(Tn intH)       -- (UintH);
\draw(Tn intH)       --(Tnint)--(UintH);
\end{tikzpicture}
\end{center}
Finally, from Equation \eqref{two:indexes} and the above discussion it follows that
\begin{equation*}
[T_n : U] \geq [T_n( \bar \phi, \pi U):  \pi U] [T_n(\phi, U\cap H):U\cap H].
\end{equation*} 
Applying $\log$, dividing by $n$, and taking the limit for $n \rightarrow \infty$ we conclude that, for every $U\in \mathcal N(G)$,
\begin{equation} \label{first:half:add:thm:eq1}
H_{alg} (\phi, U) \geq H_{alg} (\bar \phi, \pi U) + H_{alg} (\phi \restriction_H, U \cap H).
\end{equation}

Let $U_1, U_2 \in \mathcal N(G)$, and $U = U_1 U_2 \in \mathcal N(G)$. Then $\pi U \geq \pi U_1$ are elements of $\mathcal N(G/H)$, and $U \cap H \geq U_2 \cap H$ are in $\mathcal N(H)$, so that
\begin{gather}	
H_{alg} (\bar \phi, \pi U) \geq H_{alg} (\bar \phi, \pi U_1), \label{first:half:add:thm:eq2}\\
H_{alg} (\phi \restriction_H, U \cap H) \geq H_{alg} (\phi \restriction_H, U_2 \cap H). \label{first:half:add:thm:eq3}
\end{gather}
From Equations \eqref{first:half:add:thm:eq1}, \eqref{first:half:add:thm:eq2} and \eqref{first:half:add:thm:eq3}, it follows that
\begin{equation*}
h_{alg} (\phi) \geq H_{alg} (\phi, U) \geq H_{alg} (\bar \phi, \pi U_1) + H_{alg} (\phi \restriction_H, U_2 \cap H).
\end{equation*}
Taking the suprema over $U_1, U_2 \in \mathcal N(G)$, we conclude by applying Corollary \ref{cor:halg:induced:maps}.
\end{proof}

We are now in position to prove the Addition Theorem.

\begin{proof}[\bf Proof of Theorem \ref{addthm}]
In view of Proposition \ref{first:half:add:thm}, it only remains to prove that $h_{alg} (\phi) \leq h_{alg} (\bar \phi) + h_{alg}(\phi\restriction_H~).$

Let $O$ be an arbitrary open normal subgroup of $G$ such that $\phi^{-1}O\leq O$ and $[O:\phi^{-1}O]< \infty$.
Then $\phi^{-1}O$ is normal in $G$ (so in particular in $O$ itself), and
\begin{equation}\label{twoind}
[O:\phi^{-1}O] = [O : (\phi^{-1}O)(O\cap H)] \cdot [(\phi^{-1}O)(O\cap H) : \phi^{-1}O].
\end{equation}
\begin{center}
\begin{tikzpicture}
\node(O) at(0,0)     {$O$};
\node(sum)    at (0,-2)     {$(\phi^{-1}O)(O\cap H)$};
\node(phimO)  at (-2,-4)     {$\phi^{-1}O$};
\node(OintH)    at (2,-4)    {$O \cap H$};
\node (int) at (0,-6) {$\phi^{-1}O \cap H$};
\node(H)   at (4,-2)    {$H$};	
\draw(O)  -- (sum) -- (phimO) -- (int) -- (OintH) -- (sum);
\draw(H)   -- (OintH);
\end{tikzpicture}
\end{center}

First observe that since $H$ is $\phi$-stable with  $\ker\phi\leq H$, then we also have $\phi^{-1}H=H$ by Lemma \ref{lem:stable}, so 
$$\phi^{-1}(O)\cap H = \phi^{-1}O\cap \phi^{-1}H = \phi^{-1}(O\cap H).$$
Then, computing the second index in the right hand side of Equation \eqref{twoind} we obtain 
\begin{equation*}
[(\phi^{-1}O)(O\cap H) : \phi^{-1}O] = [O\cap H:\phi^{-1}(O)\cap H] =[O\cap H:\phi^{-1}(O\cap H)].
\end{equation*} 
Note that $H$ is a strongly compactly covered group, and having an open normal subgroup $O\cap H$ such that $\phi^{-1}(O\cap H)\leq O\cap H$, and $[O\cap H:\phi^{-1}(O\cap H)]<\infty$ by  Equation \eqref{twoind}. By Proposition \ref{LFcoralg},
\begin{equation}\label{2:add:th:eq:1}
h_{alg} (\phi \restriction_H) \geq \log [O\cap H:\phi^{-1}(O\cap H)].
\end{equation}
To compute the first index in the right hand side of Equation \eqref{twoind}, first note that $$(\phi^{-1}O)(O\cap H) = (\phi^{-1}(O) H )\cap O$$ by the modular law. Then, chasing the diagram 
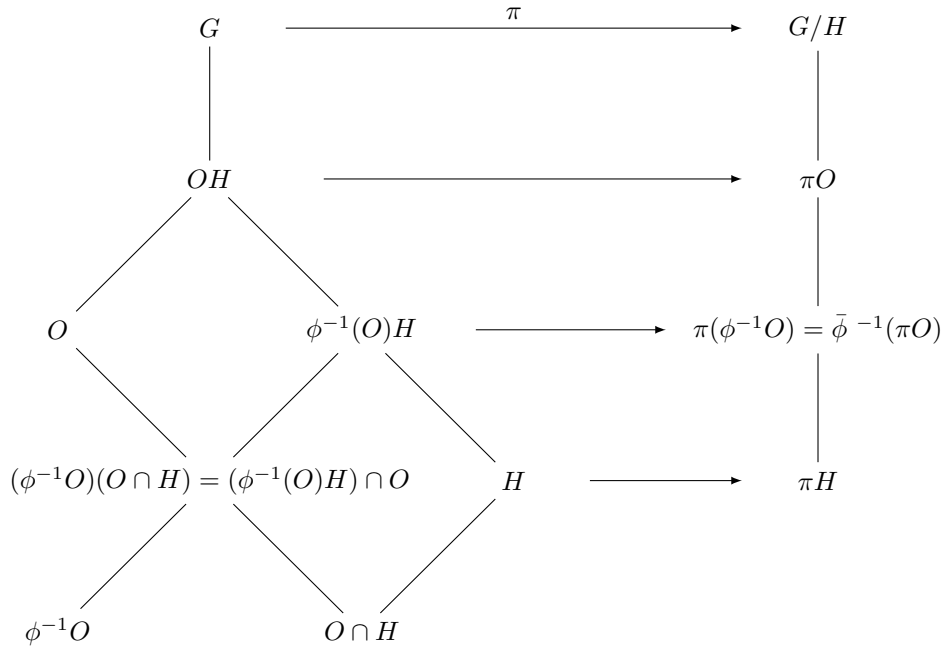
\begin{center}
\begin{tikzpicture}
\node(O) at(-2,0)     {$O$};		
\node(sum)    at (0,-2)  {$(\phi^{-1}O)(O\cap H) = (\phi^{-1}(O) H )\cap O$};
\node(phimO)  at (-2,-4)     {$\phi^{-1}O$};
\node(OintH)    at (2,-4)    {$O \cap H$};
\node(H)   at (4,-2)    {$H$};										\node(zero) at (8,-2) {$\pi H$};
\node(phimO+H) at (2,0) {$\phi^{-1}(O) H$};					\node(piphimO) at (8,0) {$\pi (\phi^{-1}O) = {\bar \phi}^{\ -1}(\pi O)$};
\node(O+H) at (0,2) {$OH$};											\node(piO+H) at (8,2) {$\pi O$};
\node(G) at (0,4) {$G$};												\node(G/H) at (8,4) {$G/H$};
\draw(O+H) -- (O)  -- (sum) -- (OintH) --(H) -- (phimO+H) -- (sum) -- (phimO);
\draw(phimO+H) --  (O+H) -- (G);
\draw(G/H) -- (piO+H) -- (piphimO) -- (zero);
				\node[above] at (4,4) {$\pi$};
\draw[-latex](1,4)--(7,4);
\draw[-latex](1.5,2)--(7,2);
\draw[-latex](3.5,0)--(6,0);
\draw[-latex](5,-2)--(7,-2);
\end{tikzpicture}
\end{center}
of the subgroups of $G$ and $G/H$, one can easily verify that
\begin{multline*}
[O:(\phi^{-1}O)(O\cap H)] = [O : (\phi^{-1}(O) H )\cap O] = [OH:\phi^{-1}(O)H] = \\ = [\pi(OH) : \pi ( \phi^{-1}(O)H )] = [\pi O:\pi (\phi^{-1}O)] = [\pi O: {\bar \phi}^{\ -1}(\pi O)].
\end{multline*}

By our assumption, $G/H$ is a strongly compactly covered group. As $\pi O$ is an open normal subgroup of $G/H$, such that ${\bar \phi}^{\ -1}(\pi O)\leq \pi O$ and $[\pi O:{\bar \phi}^{\ -1}(\pi O)]<\infty$ by Equation \eqref{twoind}, Proposition \ref{LFcoralg} applied to $G/H$ and to $\bar \phi$ implies that
\begin{equation}\label{2:add:th:eq:2}
h_{alg} (\bar \phi) \geq \log [\pi O:{\bar \phi}^{\ -1}(\pi O)].
\end{equation}
Summing up Equation \eqref{2:add:th:eq:1} and Equation \eqref{2:add:th:eq:2}, and using Equation \eqref{twoind}, we obtain 
\[h_{alg} (\bar \phi)+h_{alg} (\phi \restriction_H) \geq \log [O:\phi^{-1}O].\] 
By the arbitrariness of $O$ and applying Proposition \ref{LFcoralg} to $G$ we conclude that $h_{alg} (\bar \phi) + h_{alg} (\phi \restriction_H) \geq h_{alg} (\phi)$.
\end{proof}

\begin{remark}
For the subclass of compactly covered locally compact abelian  groups  Theorem  \ref{addthm} admits an alternative proof. Indeed, one can  combine the Addition Theorem proved in \cite{GBV} for the topological entropy of continuous endomorphisms of locally compact totally disconnected groups with the Bridge Theorem proved in \cite{DGB} (see Theorem \ref{BT}).
\end{remark}

The next corollary shows that when we compute the algebraic entropy of a topological \emph{automorphism}  of a strongly compactly covered group $G$, we may assume that $G$ is also totally disconnected, i.e., its connected component $c(G)$ is trivial. 

\begin{corol}\label{c(G)}
Let $G$ be a strongly compactly covered group, and $\phi \in \End(G)$ be such that $c(G)$ is $\phi$-stable and contains $\ker \phi$.  Then  $h_{alg} (\phi) =h_{alg} (\bar \phi)$, where $\bar \phi : G/c(G) \to G/c(G)$ is the induced map. 

In particular, if $\phi\in \Aut(G)$,  then $c(G)$ is $\phi$-stable and $h_{alg} (\phi) =h_{alg} (\bar \phi)$.
\end{corol}
\begin{proof}
As $G$ is compactly covered by Proposition \ref{cclca:structure}, the connected component $c(G)$ is compact by \cite[Lemma 2.15]{BWY}, so
$h_{alg} (\phi \restriction_{c(G)})=0$. Applying Theorem \ref{addthm} we get \[h_{alg} (\phi) =h_{alg} (\bar \phi) + h_{alg} (\phi \restriction_{c(G)})=h_{alg} (\bar \phi).\]  For the last assertion use the fact that $c(G)$ is a characteristic subgroup of $G$. 
\end{proof}

\subsection*{Acknowledgments}
This work is supported by Programma SIR 2014 by MIUR, project GADYGR, number RBSI14V2LI, cup G22I15000160008 and by INdAM - Istituto Nazionale di Alta Matematica.

\end{document}